\documentclass[12pt]{article}
\usepackage{epsfig}
\usepackage{amsfonts}
\usepackage{amssymb}
\usepackage{amstext}
\usepackage{amsmath}
\usepackage{xspace}
\usepackage{theorem}
\usepackage{graphicx}
\makeatletter
\newenvironment{proof}{{\bf Proof:  }}{\hfill\rule{2mm}{2mm}}

\newcommand{\junk}[1]{}
\newtheorem{theorem}{Theorem}
\newtheorem{lemma}[theorem]{Lemma}
\newtheorem{conjecture}{Conjecture}

\newtheorem{corollary}[theorem]{Corollary}
\newtheorem{definition}{Definition}
\newtheorem{remark}{Remark}
\newtheorem{question}{Question}

\newcommand{\Aut}{\ensuremath{\textnormal{Aut}}}
\newcommand{\spec}{\ensuremath{\text{spec}}}

\newcommand{\Th}{\ensuremath{\text{Th}}}

\DeclareMathOperator{\Tr}{Tr}
\title{Applications of the Harary-Sachs Theorem for Hypergraphs}
\author{Gregory J. Clark and Joshua N. Cooper\\
\small Sa\"id Business School\\[-0.8ex]
\small University of Oxford\\
\small \texttt{gregory.clark@sbs.ox.ac.uk  }\\
\small Department of Mathematics\\[-0.8ex]
\small University of South Carolina\\
\small \texttt{cooper@math.sc.edu}\\
}
\begin{document}
\maketitle
\begin{abstract}
The Harary-Sachs theorem for $k$-uniform hypergraphs equates the codegree-$d$ coefficient of the adjacency characteristic polynomial of a uniform hypergraph with a weighted sum of subgraph counts over certain multi-hypergraphs with $d$ edges. We begin by showing that the classical Harary-Sachs theorem for graphs is indeed a special case of this general theorem.  To this end we apply the generalized Harary-Sachs theorem to the leading coefficients of the characteristic polynomial of various hypergraphs. In particular, we provide explicit and asymptotic formulas for the contribution of the $k$-uniform simplex to the codegree-$d$ coefficient.  Moreover, we provide an explicit formula for the leading terms of the characteristic polynomial of a 3-uniform hypergraph and further show how this can be used to determine the complete spectrum of a hypergraph.    We conclude with a conjecture concerning the multiplicity of the zero-eigenvalue of a hypergraph.
\end{abstract}
\section{Introduction}
A foundational result in spectral graph theory of Harary \cite{Har} (and more explicitly, Sachs \cite{Sac}) showed how the coefficients of a graph's adjacency characteristic polynomial can be expressed as a weighted sum of the counts of certain subgraphs.   Famously, the Harary-Sachs theorem yields the leading coefficients of $\phi(G) = \sum_{i=0}^nc_i\lambda^{n-i}$: $c_0 = 1$, $c_1 = 0$, $c_2 = |E|$, and $c_3 = -2(\# \text{ of triangles in $G$})$.  These subgraphs are referred to as \emph{elementary subgraphs} of $G$.  An elementary subgraph of a graph is a simple subgraph whose components are edges or cycles (for further details, see \cite{Big}, Chapter 7).  In keeping with their notation, let $\Lambda_d$ be the set of simple graphs with $d$ vertices whose components are edges or cycles.
\begin{theorem}
\label{T:Harary}
(Harary-Sachs Theorem \cite{Har},\cite{Sac}) Let $G$ be a labeled simple graph on $n$ vertices.  Further let $c_d$ denote the codegree-$d$ coefficient of $\lambda^{n-d}$ in the characteristic polynomial of $G$, then 
\[
c_{d} = \sum_{H \in \Lambda_d} (-1)^{c(H)}2^{z(H)} [\# H \subseteq G]
\]
where $c(H)$ is the number of components of $H$, $z(H)$ is the number of components which are cycles, and $[\# H \subseteq G]$ denotes the number of (labeled) subgraphs of $G$ which are isomorphic to $H$.
\end{theorem} 

An analogous result for the characteristic polynomial of a hypergraph was proven in \cite{Cla0} and we maintain much of the notation and nomencalture therein.  In the same way that the coefficients of the characteristic polynomial of a graph can be expressed in terms of elementary subgraphs, the coefficients of the characteristic polynomial of a hypergraph can be expressed in terms of \textit{Veblen hypergraphs}.

\begin{definition}
A hypergraph is \emph{$k$-uniform} if each edge contains precisely $k$ vertices.  We will refer to $k$-uniform hypergraphs as $k$-graphs.  Moreover, a hypergraph is \emph{$k$-valent} if $k$ divides the degree of each vertex.  A \emph{Veblen hypergraph} is a $k$-uniform, $k$-valent multi-hypergraph.  When $k = 2$ we simply write Veblen graph. 
\end{definition}

Note that a Veblen graph is not necessarily an elementary subgraph (and vice versa).  For example, the 2-cycle (viewed as a multigraph) is not an elementary graph as it is not simple; moreover, the single edge is not a Veblen graph because it is not $2$-valent.  However, these two graphs are related under the \textit{flattening} operation (i.e., the flattening of a 2-cyle is an edge).

\begin{definition}
For a labeled multi-hypergraph $H$, we call  the simple $k$-graph formed by removing duplicate edges of $H$ the \emph{flattening} of $H$ and denote it $\underline{H}$. We say that $H$ is an \emph{infragraph} of ${\cal H}$ if $\underline{H} \subseteq {\cal H}$.  Let ${\cal V}_d({\cal H})$ denote the set of isomorphism classes of connected, labeled Veblen infragraphs with $d$ edges of ${\cal H}$.  We further let ${\cal V}_d^*({\cal H})$ denote the set of isomorphism classes of possibly disconnected, labeled Veblen infragraphs with $d$ edges.
\end{definition}

The proof of the Harary-Sachs theorem makes use of the permutation definition of the determinant to determine which subgraphs have a non-zero contribution to the determinant of the adjacency matrix.  Intuitively, this can be thought of as following disjoint closed walks.  This intuition doesn't translate naturally to the language of hypergraphs, \textit{a priori}; however, it motivates the following definition of a \textit{rooting of a $k$-graph}.

\begin{definition}
The \emph{$u$-rooted directed star of a $k$-uniform edge $e$} is 
\[
S_e(u) = (e,\{uv: v \in e, u \neq v\}).
\]
A \emph{rooting} of a $k$-graph ${\cal H}$ is an ordering $R = (S_{e_{1}}(v_{1})
, S_{e_{2}}(v_{2}), \dots, S_{e_{m}}(v_m))$ such that $E({\cal H}) = \{e_1, \dots, e_m\}$ and $v_i \leq v_{i+1}$.  Given a rooting of $\cal H$ we define the \emph{rooted multi-digraph} of $R$ to be
\[
D_R = \bigcup_{i=1}^m S_{e_i}(v_i)
\]
where the union sums edge multiplicities.  We say that a rooting $R$ is an \emph{Euler rooting} if $D_R$ is Eulerian.  We denote the mutli-set of Eulerian rooted digraphs of ${\cal H}$ as $R({\cal H})$.
\end{definition}

We are now able to determine to what extent each Veblen hypergraph contributes to a particular coefficient in the characteristic polynomial of a given hypergraph.  

\begin{definition}
The \emph{associated coefficient of a connected Veblen hypergraph} $H$ is 
\[
C_H = \sum_{D \in R(H)} \left(\frac{\tau_{D}}{\prod_{v \in V(D)} \deg^-(v)}\right)
\]
 where $\tau(G)$ is the number of arborescences (i.e., the number of rooted subtrees of $G$ with a specified root). The \emph{associated coefficient} of a (possibly disconnected) Veblen hypergraph $H = \bigcup_{i=1}^m G_i$ is 
\[
C_H = \prod_{i=1}^m C_{G_i}.
\]
\end{definition}

The concept of an \emph{associated coefficient of a hypergraph} first appeared in \cite{Coo}.  There, the authors provide a combinatorial description of the codegree-$k$ and codegree-$(k+1)$ coefficient, denoted $c_k$ and $c_{k+1}$ respectively, for the normalized adjacency characteristic polynomial of a $k$-uniform hypergraph. 
\begin{theorem}
\label{T:Codegree-k}
\cite{Coo} Let ${\cal H}$ be a $k$-graph.  Then
\[
c_k=-k^{k-2}(k-1)^{n-k}|E({\cal H})|
\]
and
\[
c_{k+1} = -C_k (k-1)^{n-k}(\#\text{ of simplices in ${\cal H}$}),
\]
 where $C_k$ is some constant depending on $k$.
\end{theorem}
This idea was further studied by Shao, Qi, and Hu where the authors prove (restating Theorem 4.1 of \cite{Sha}),
\[
c_d = (k-1)^{n-1} \sum_{D \in {\bf D}} f_D|\mathfrak{E}(D)|
\]
where ${\bf D}$ is a certain large family of digraphs, $f_D$ is a function of $D$ and $\mathfrak{E}(D)$ is the set of Euler circuits in $D$.  The authors then use their formula to provide a general description of $\Tr_2(T)$ and $\Tr_3(T)$ for a general tensor $T$. 

The enterprise of this paper is to continue this line of research by determining explicit and asymptotic formulas for $C_k$ as well as formula for $Tr_2(T), \dots, Tr_6(T)$ where $T$ is the normalized adjacency hypermatrix of a hypergraph.  To do so, we leverage the generalized Harary-Sachs theorem.

\begin{theorem}(\cite{Cla0})
\label{T:codegreeFormula}
For a simple $k$-graph $\cal H$, 
\[
c_d = \sum_{H \in {{\cal V}^*_d}({\cal H})}(-(k-1)^n)^{c(H)}C_H(\# H \subseteq {\cal H})
\]
where $c(H)$ is the number of components of $H$, $C_H$ is a constant depending only on $H$, and $(\# H \subseteq {\cal H})$ is the number of times $H$ occurs (in a certain sense that is a minor generalization of the subgraph relation) in ${\cal H}$.
\end{theorem}


Finally, we formalize how we enumerate the occurrences of a given infragraph within a hypergraph.

\begin{definition}
  For a $k$-graph ${\cal H}$ and a Veblen $k$-graph $H = \bigcup_{i=1}^m G_i$ we define 
\[
(\#H \subseteq {\cal H}) =\frac{1}{|\Aut(H)|}\prod_{i=1}^m|\Aut(\underline{G_i})||\{S \subseteq {\cal H} : S \cong \underline{G_i}\}|.
\]
\end{definition}

The paper is arranged as follows. We begin by showing that Theorem \ref{T:codegreeFormula} is a faithful generalization of the Harary-Sachs theorem in the sense that it yields the same conclusion as the traditional Harary-Sachs theorem in the case when $k=2$.  Using this theorem we answer a question of $\cite{Coo}$ by providing a computationally efficient formula for $C_H$ in the case where $H$ is the $k$-uniform simplex, i.e., $K_{k+1}^{(k)}$.  We further answer a question of \cite{Hu2} by applying our explicit formula for higher order traces to the computation of the characteristic polynomial of a hypergraph.  Moreover, we provide an example of how the leading coefficients of $\phi({\cal H})$ can be used to determine all of its coefficients.  We conclude with a conjecture concerning the multiplicity of the zero-eigenvalue.

\section{Deducing the Harary-Sachs Theorem for Graphs}

With Theorem \ref{T:codegreeFormula} in hand we can express the codegree-$d$ coefficient of the normalized adjacency characteristic polynomial of a hypergraph as a weighted sum of Veblen infragraphs with $d$ \emph{edges}.  In the case when $G$ is a 2-graph our theorem simplifies to
\[
c_d = \sum_{H \in {{\cal V}^*_d}}(-1)^{c(H)}C_H(\# H \subseteq G).
\]
Note that the Harary-Sachs theorem (Theorem \ref{T:Harary}) expresses the codegree-$d$ coefficient as a weighted sum over certain subgraphs on $d$ \emph{vertices}. 

Recall that an \emph{elementary subgraph} of a graph $G$ is a simple subgraph of $G$ whose components are edges or cycles (see \cite{Big} for further details). In keeping with their notation we let $\Lambda_d(G)$ be the set of elementary subgraphs of $G$ with $d$ vertices.  Observe that a connected elementary graph is the flattening of a cycle (e.g., the flattening of a 2-cycle is an edge).  Note that cycles (and disjoint unions of cycles) are the only two regular non-empty graphs which have an equal number of vertices and edges.  Indeed
\[
\Lambda_d(G) \subseteq \{\underline{H} : H \in {\cal V}^*_d(G)\}.
\]
By straightforward computation we have that the associated coefficient of a 2-cycle is $1$ and the associated coefficient of a simple cycle (i.e., any cycle which is not a 2-cycle) is $2$.  Restricting our attention to $\Lambda_d(G)$, we have by Theorem \ref{T:codegreeFormula}
\begin{equation}
\label{E:Harary}
\sum_{H \in \Lambda_d(G)}(-1)^{}C_H(\#H \subseteq G) = (-1)^{c(H)}2^{z(H)}(\#H \subseteq G)
\end{equation}
where $z(H)$ is the number of cycles in $H$.  Note that Equation \ref{E:Harary} is the conclusion of the Harary-Sachs theorem.  We deduce the Harary-Sachs theorem from Theorem \ref{T:codegreeFormula} by showing that the summands of
\[
c_d = \sum_{H \in {{\cal V}^*_d}}(-1)^{c(H)}C_H(\# H \subseteq G)
\]
which \emph{do not arise from elementary graphs} sum to zero.  We make this statement precise with the following.

\begin{definition}
For a multigraph $G$ and an edge $e \in E(G)$, write $m(e) = m_G(e)$ for the multiplicity of $e$ in $E(G)$.  Let $G$ be a connected, labeled Veblen graph with distinguishable multi-edges.  Given a multiset $P$ of multigraphs whose multi-edges are indistinguishable, each on the vertex set $V(G)$, we write $P \vdash G$ if $P$ partitions the edges of $G$ in the sense that for each $e \in E(G)$
\[
\sum_{P_i \in P} m_{P_i}(e) = m_G(e).
\]
\end{definition}

\begin{lemma}
\label{L:HS}
 \begin{displaymath}
   \sum_{P \vdash G}(-1)^{c(P)}C_P =  \left\{
     \begin{array}{ll}
       1 & :\text{$G$ is a 2-cycle} \\
       2 & :\text{$G$ is a simple cycle} \\
       0 & :\text{otherwise.}
     \end{array}
   \right.
\end{displaymath} 
\end{lemma}

\begin{remark}
We refer to Veblen 2-graphs as Veblen graphs.  An \emph{Euler orientation} of a graph $G$ is an orientation of the edges of $G$ such that the resulting digraph is Eulerian. 
\end{remark} 

Before proving Lemma \ref{L:HS} we provide a combinatorial formula for the associated coefficient of a Veblen graph. 

\begin{lemma}
\label{L:graphcoef}
Let $G$ be a connected Veblen graph and let $\mathfrak{E}(G)$ be the set of Euler circuits of $G$.  We have
\[
C_G = \frac{|\mathfrak{E}(G)|}{\prod_{e\in E(G)}m(e)!}.
\]
\end{lemma}

Our proof of Lemma \ref{L:graphcoef} makes use of the BEST Theorem.

\begin{theorem}
\label{T:BEST}
(BEST Theorem, \cite{Aar}) The number of Euler circuits in a connected Eulerian graph $G$ is 
\[
|\mathfrak{E}(G)| = \tau(G) \prod_{v \in V} (\deg(v) -1)!
\]
 where $\tau(G)$ is the number of arborescences (i.e., the number of rooted subtrees of $G$ with a specified root).
\end{theorem}

We now prove Lemma \ref{L:graphcoef}.

\begin{proof}
Let $G$ be a Veblen graph.  Since $G$ is Eulerian, we write $\deg(v) = 2d_v$ for convenience.  Because $\deg^-(v) = d_v$ for all $D \in R(G)$ we have
\[
C_G = \frac{1}{\prod_{v \in V(G)}d_v} \sum_{D \in R(G)}\tau_{D}.
\]
By the BEST theorem (i.e, Theorem \ref{T:BEST}) we have 
\[
\tau_D = \frac{|\mathfrak{E}(D)|}{\prod_{v \in V(G)}(d_v-1)!}.
\]
Let $N_D(v)$ denote the out-neighborhood of $v$ in $D$ and let $\deg_D(v,u)$ denote the number of edges directed from $v$ to $u$ in $D$.  We denote
\[
\binom{d_v}{N_D(v)} = \frac{d_v!}{\prod_{u \in N_D(v)}\deg_D(v,u)!}
\]
which is the number of linear orderings of out-edges of $v$ in an Euler orientation $D$.  Consider the equivalence relation $\sim_{R(G)}$ where $R \sim R'$ if and only if $D_{R} = D_{R'}$.  Note that $\sim$ identifies two Euler rootings if their associated digraphs are the same Euler orientation.  Let $[R]$ denote the equivalence class of $R$ under $\sim$.  Suppose $R(G) = \bigcup_{i=1}^t[R_i]$ and note 
\[
|[R]| = \prod_{v \in V(G)} \binom{d_v}{N_{D_R}(v)}
\]
as two rootings in $[R]$ differ only in the ordering of the $u$-rooted stars for $u \in V(G)$.
Let ${\cal O}(G) = \{D_{1}, \dots, D_{t}\}$, where $D_i = D_{R_i}$, denote the Euler orientations of $G$.  For convenience, we write $N_i$ for $N_{D_i}$.  We equate
\[
\sum_{D \in R(G)} \tau_D = \sum_{i=1}^t \tau_{i} |[R_i]| = \sum_{i=1}^t \left(\tau_i \prod_{v \in V(G)} \binom{d_v}{N_{i}(v)} \right).
\]
Substitution and simplification yields
\[
C_G = \sum_{i =1}^t \frac{|\mathfrak{E}(D_i)|}{\prod_{v \in V(G)}\left(\prod_{u \in N_i(v)} \deg_i(v,u)!\right)}.
\]
Since $\deg_i(u,v) + \deg_i(v,u) = m(uv)$ we have
\[
\binom{m(uv)}{\deg_i(u,v)} = \frac{m(uv)!}{\deg_i(u,v)!\deg_i(v,u)!} 
\]
so that
\begin{align*}
\frac{\prod_{e \in E(G)}m(e)!}{\prod_{v \in V(G)} \prod_{u \in V(G)} \deg_i(v,u)!} &= \frac{\prod_{e \in E(G)}m(e)!}{\prod_{u < v \in V(G)} \deg_i(u,v)!\deg_i(v,u)!}\\
&= \prod_{uv \in E(G), u < v} \binom{m(uv)}{\deg_i(u,v)}.
\end{align*}
Then
\begin{align*}
C_G &= \sum_{i=1}^t \frac{|\mathfrak{E}(D_i)|}{\prod_{v \in V(G)}\left(\prod_{u \in N_i(v)} \deg_i(v,u)!\right)} \\
&= \frac{\prod_{e \in E(G)}m(e)!}{\prod_{e \in E(G)} m(e)!}\sum_{i=1}^t \frac{|\mathfrak{E}(D_i)|}{\prod_{v \in V(G)}\left(\prod_{u \in N_i(v)} \deg_i(v,u)!\right)}\\
&= \frac{1}{\prod_{e \in E(G)}m(e)!}\sum_{i=1}^t \left(\prod_{uv \in E(G), u < v}\binom{m_G(uv)}{\deg_i(u,v)}\right)|\mathfrak{E}(D_i)|\\
&=\frac{|\mathfrak{E}(G)|}{\prod_{e \in E(G)}m(e)!}
\end{align*}
where the last equality follows from the observation that $D_i$ has indistinguishable multi-edges.
\end{proof}

We now prove Lemma \ref{L:HS} with the aid of Theorem \ref{T:Veblen}.

\begin{theorem}
\label{T:Veblen}
(Veblen's Theorem, \cite{Veb}) The set of edges of a finite graph can be written as a union of disjoint simple cycles if and only if every vertex has even degree.
\end{theorem}

\begin{proof}
Let $G$ be a connected Veblen graph which is not a cycle.  Further assume that the multi-edges of $G$ are distinguishable.  We aim to show
\[
\sum_{P \vdash G}(-1)^{c(P)}C_P = 0.
\]
By Lemma \ref{L:graphcoef} we have for connected $G$, 
\[
C_G = \frac{|\mathfrak{E}(G)|}{\prod_{e\in E(G)}m(e)!}.
\]
Let $P = \bigcup_{i=1}^t P_i$ be a disjoint union of Veblen graphs.  We denote
\[
\binom{m_G(e)}{P(e)} = \frac{m_G(e)!}{\prod_{i=1}^tm_{P_i}(e)!}
\]
and
\[
|\mathfrak{E}(P)| = \prod_{i=1}^t|\mathfrak{E}(P_i)|.
\]
We equate
\[
\left(\prod_{e \in E(G)}m(e)!\right)\sum_{P \vdash G}(-1)^{c(P)}C_P  = \sum_{P \vdash G} (-1)^{c(P)}|\mathfrak{E}(P)| \prod_{e \in E(G)} \binom{m_G(e)}{P(e)}.
\]
Notice that $|\mathfrak{E}(P)| \prod_{e \in E(G)} \binom{m_G(e)}{P(e)}$ counts the number of partitions of $E(G)$ into edge-disjoint Euler circuits of graphs on $V(G)$ with unlabelled edges which are precisely the elements of $P$.

We say that an Euler circuit is \emph{decomposable} if it can be written as a union of (at least) two edge disjoint Euler circuits, and \emph{indecomposable} otherwise.  The number of decompositions of Euler circuits of $G$ into exactly $t$ parts is
\[
\sum_{P \vdash G, }\left(|\mathfrak{E}(P)| \prod_{e \in E(G)} \binom{m_G(e)}{P(e)}\right).
\]
By Inclusion/Exclusion, the number of indecomposable Euler circuits of $G$ is
\[
\sum_{t=1}(-1)^t \sum_{P \vdash G, }\left(|\mathfrak{E}(P)| \prod_{e \in E(G)} \binom{m_G(e)}{P(e)}\right).
\]
We assumed that $G$ is a Veblen graph which is not a cycle.  We have by Veblen's theorem (i.e., Theorem \ref{T:Veblen}) that every Euler circuit in $G$ is decomposable.  It follows that $G$ has no indecomposable Euler circuits; that is to say,
\[
\sum_{t=1}(-1)^t \sum_{P \vdash G, }\left(|\mathfrak{E}(P)| \prod_{e \in E(G)} \binom{m_G(e)}{P(e)}\right) = 0,
\]
from which the desired conclusion follows. 
\end{proof}

\section{The Associated Coefficient of a $k$-uniform Simplex and the Codegree-$k$ Coefficient}
 
 Fix $k \geq 2$ and let 
 \[
 K_{k+1}^{(k)} = \left([k+1], \binom{[k+1]}{k}\right)
 \] 
 be the $k$-uniform simplex.  It was shown in \cite{Coo} that the $k$-uniform simplex is the only connected Veblen hypergraph with $k+1$ edges, up to isomorphism. Moreover, it was shown for a $k$-graph ${\cal H},$ 
 \[
 c_{k+1}({\cal H}) = -(k-1)^{n-k}C_k(\# H \subseteq {\cal H})
 \]
for some constant $C_k$ depending only on $k$.  The authors of \cite{Coo} were able to show that $C_2 = 2$, $C_3 = 21$, $C_4 = 588$, $C_5 = 28230$ via laborious use of resultants.  In this section we provide an explicit, efficient formula for $C_k$ and use it to compute $C_{100}$ (see A320653 \cite{OEIS}).

\begin{remark} 
For the remainder of this section let $H$ denote the $k$-uniform simplex $K_k^{(k+1)}$.  We will denote an arbitrary $k$-graph using ${\cal H}$.
\end{remark}
Let ${\mathfrak D}_{k+1}$ be the set of derangements of $[k+1]$, i.e., permutations without any fixed points.  We begin by showing that the Euler rootings of the $k$-uniform simplex are in bijection with the derangements of $[k+1]$. 
\begin{lemma}
For $\sigma \in \mathfrak{D}_{k+1}$ define
\[
D_{\sigma} = \bigcup_{i=1}^{k+1}S_{[k+1] \setminus \{\sigma(i)\}}(i).
\]
We have that
\[
R(H) = \{D_\sigma : \sigma \in \mathfrak{D}_{k+1}\}.
\]
\end{lemma}

\begin{proof}
Let $\sigma \in {\mathfrak D}_{k+1}$ and consider $D_\sigma$.  We suppress the subscript and write $D$ for convenience.  Note that $D$ is Eulerian whence  
\[
\deg^+_D(j) = (k-1)|\{i : v_i = j\}| = k-1 = |\{i : v_i \neq j, j \in e_i\}| = \deg^-_D(j)
\]
for all $j \in [k+1]$ as $\sigma$ is a derangement and there are exactly $k$ edges containing $j$. 

Suppose $D_R \in R(H)$.  As $D$ is Eulerian and each vertex is contained in exactly $k$ edges it must be the case that each vertex is chosen as a root in $R$ precisely once.  It follows that
\[
D = \bigcup_{i=1}^{k+1} S_{[k+1]\setminus \{v_i\}}(i)
\]
where $\{v_1, \dots, v_{k+1}\} = [k+1]$.  Observe that $\sigma = \{(i,v_i)\}_{i=1}^{k+1} \in \mathfrak{D}_{k+1}$ is a derangement whence $v_i \neq i$.  The conclusion follows from the fact that $D = D_\sigma$. 
\end{proof}

This immediately implies the following.  

\begin{lemma}
For the $k$-uniform simplex $H$,
\[
C_H = \sum_{\sigma \in {\mathfrak D}_{k+1}}\frac{\tau_\sigma}{\prod_{v \in V(D_\sigma)} \deg^-(v)}.
\]
\end{lemma}

We now show that a summand in the aformentioned formula of $C_H$ depends only on the cycle type of the derangement.  

\begin{theorem}
\label{T:simplexFormula}
Let $\sigma \in {\mathfrak D}_{k+1}$ with cycle decomposition $c_1c_2\dots c_t$ where cycle $c_i$ has length $\ell_i$.  Then

\[
C_H = \frac{1}{(k-1)^{k+1}(k+1)} \sum_{\sigma = c_1c_2\dots c_t \in {\mathfrak D}_{k+1}} \prod_{i=1}^t\left( k^{\ell_i} + (-1)^{\ell_i + 1}\right).
\]
\end{theorem}

We first prove a technical lemma.  The notation $\spec(M)$ denotes the ordinary (multiset) spectrum of a matrix $M$. 
\begin{lemma}
\label{L:permutationSpectrum}
For $\sigma \in S_{n+1}$, $spec(M_\sigma - J) = (\spec(M_\sigma) \setminus \{1\}) \cup \{-n\}$ where $M_\sigma$ is the permutation matrix associated with $\sigma$.
\end{lemma}

\begin{proof}
Let $\sigma$ be a permutation of $[n+1]$ with cycles $c_1, c_2, \dots, c_t$ of length $l_1, l_2, \dots, l_t$, respectively.  Recall that the spectrum of $M_\sigma$ is given by
\[
\spec(M_\sigma) = \bigcup_{i =1}^t \{\zeta_{l_i}^0, \zeta_{l_i}^1, \dots, \zeta_{l_i}^{l_i-1}\}
\]
and note that the spectrum of $M_{\sigma}$ depends only on the cycle type of $\sigma$.  Without loss of generality, suppose that the cycles of $\sigma$ are increasing (i.e., $[1,\ell_1], [\ell_1+1, \ell_2+\ell_1], \dots$).  Consider the following block partition
\[
 M_\sigma - J_{n+1} = \left( \begin{array}{cccc}
B_1 & -J_{\ell_2} & \dots &-J_{\ell_t} \\
-J_{\ell_1} & B_2 & \dots & -J_{\ell_n} \\
\vdots & \vdots & \ddots & \vdots\\
-J_{\ell_1} & -J_{\ell_2} & \dots & B_t \end{array} \right)
\]
 where $B_i$ is the $l_i \times l_i$ square circulant matrix corresponding to $c_i$,
 \[
 B_i = \left( \begin{array}{ccccc}
-1 & 0 & -1 & \dots &-1 \\
-1 & -1& 0 & \dots  & -1 \\
\vdots & \vdots & \vdots &\ddots & \vdots\\
-1 & -1 & -1 &\dots & 0\\
0 & -1 & -1 &\dots & -1 \end{array} \right).
\]

Let $j > 0$ and consider the eigenpair $(\zeta_{\ell_i}^j, {\bf x})$ of $M_{\sigma}$.  We compute
\[(M_{\sigma} - J){\bf x} = M_\sigma{\bf x} - J{\bf x} = \zeta_{\ell_i}^j {\bf x} - {\bf 0} = \zeta_{\ell_i}^j {\bf x}
\]
where $J{\bf x} = 0$ because the coordinates of $x$ corresponding to $c_i$ are the complete set of $\ell_i$-th roots of unity.

We now show $1 \in \spec(M_\sigma - J)$ has a geometric multiplicity of $t-1$. Fix  $1 \leq i \leq t-1$ and consider ${\bf x} \in \mathbb{C}^{n+1}$ where
\begin{displaymath}
   {\bf x}_j = \left\{
     \begin{array}{ll}
       1 & : j \in c_i \\
       -\frac{\ell_i}{\ell_{i+1}} & : j \in c_{i+1}\\
       0 & :  \text{otherwise.}
     \end{array}
   \right.
\end{displaymath} 
For $j \in c_i$ we have,
\[
((M_\sigma - J){\bf x})_j = -(\ell_i-1) + \frac{\ell_i}{\ell_{i+1}}(\ell_{i+1}) = 1 = {\bf x}_j,
\]
for $j \in c_{i+1}$ we have,
\[
((M_\sigma - J){\bf x})_j = -\ell_i + \frac{\ell_i}{\ell_{i+1}}(\ell_{i+1}-1) = - \frac{\ell_i}{\ell_{i+1}} = {\bf x}_j,
\]
and for $j \notin c_i, c_{i+1}$ we have
\[
((M_\sigma - J){\bf x})_j = -\ell_i + \frac{\ell_i}{\ell_{i+1}}(\ell_{i+1}) = 0 = {\bf x}_j.
\]
Therefore, $(1, {\bf x})$ is an eigenpair for $1 \leq  i \leq t$; moreover, these vectors are linearly independent.  

Finally, consider the all-ones vector ${\bf 1} \in \mathbb{C}^{n+1}$ where 
\[
(M_\sigma - J) {\bf 1} = -n{\bf 1}
\]
so $(-n, {\bf 1})$ is an eigenpair of $M_\sigma - J$.  
We have shown 
\[
\spec(M_\sigma - J) \supseteq (\spec(M_\sigma) \setminus \{1\}) \cup \{-n\}
\]
and the reverse inclusion follows from the observation that the multiplicities on the right-hand side add up to (at least, and therefore exactly) $n+1$. 
\end{proof}

We now prove Theorem \ref{T:simplexFormula}.

\begin{proof}
Consider
\[
C_H = \sum_{\sigma \in {\mathfrak D}_{k+1}}\frac{\tau_\sigma}{\prod_{v \in V(D_\sigma)} \deg^-(v)}.
\]  
As previously mentioned, each vertex of $H$ is a root of exactly one edge in an Euler rooting of $H$.  It follows that, for all $\sigma \in {\mathfrak D}_{k+1}$, \[\prod_{v \in V(D_\sigma)} \deg^-(v) = (k-1)^{k+1}.\]   Thus,
\[
C_H = \frac{1}{(k-1)^{k+1}}\sum_{\sigma \in {\mathfrak D}_{k+1}} \tau_\sigma.
\]  
Consider $\sigma = c_1c_2 \dots c_t$ where cycle $c_i$ has length $\ell_i$.  Observe that
\[
{\cal L}(D_\sigma) = kI + M_\sigma - J.
\]
Since $\spec(kI + M_\sigma - J) = k + \spec(M_\sigma - J)$ we have by Lemma \ref{L:permutationSpectrum} and Kirchoff's theorem
\[
\tau_{\sigma} = \frac{\prod_{i = 1}^t \left(k^{l_i} + (-1)^{l_i+1} \right)}{k+1}.
\]
The desired equality follows by substitution. 
\end{proof}

\begin{remark}
It was shown in \cite{Coo} that for the $k$-uniform simplex,
\[
c_{k+1} = -C_k(k-1)^{n-k},
\] and by Theorem \ref{T:codegreeFormula} we have that
\[
c_{k+1} = -C_H(k-1)^n.
\]
In our notation, we can write \(C_k = (k-1)^kC_H\).
\end{remark}
 For ease of computation we consider $C_k$ instead of $C_H$.
As it is stated, Theorem \ref{T:simplexFormula} is slow to compute as $|{\mathfrak D}_{n+1}| \sim n!/e$.  However, summing over all derangements of $[k+1]$ is wasteful, as we have shown $C_k$ is a function only of the cycle structure of $\sigma$.  In fact, 
\[
\prod_{i=1}^t \left(k^{l_i} +(-1)^{l_i+1}\right)
\] 
is constant for derangements with the same cycle type.  We present a reformulation of Theorem \ref{T:simplexFormula} which has the advantage of considering a smaller search space.

\begin{definition}
Let $P(n)$ be the set of partitions of $n$ and let $P_{\geq 2}(n) \subseteq P(n)$ be the set of partitions of $n$ into parts of size at least $2$.  For $p \in P_{\geq 2}(k+1)$, let $\mathfrak{D}_{k+1}(p) \subseteq \mathfrak{D}_{k+1}$ be the set of derangements whose cycle lengths agree with the parts of $p$.  Further, for a partition $p \in P(n)$ let $V_p :[n] \to [0,n]$ be the map $V_p(i) = |\{j : p_j = i\}|$.
\end{definition}
 We reformulate Theorem \ref{T:simplexFormula} (for $C_k$) as follows.
\begin{corollary}
\label{C:improvedCk}
\[
C_k = \frac{1}{(k-1)(k+1)}\sum_{p = (p_1, \dots, p_t) \in P_{\geq 2}(k+1)}\left( |\mathfrak{D}_{k+1}(p)|\prod_{i=1}^t\left( k^{p_i} +(-1)^{p_i+1}\right)
\right)
\]
where
\[
|{\mathfrak D}_{k+1}(p)| = \frac{(k+1)!}{\left(\prod_{i=1}^t p_i\right)\left(\prod_{i=2}^{k+1}V_p(i)\right)}.
\]
\end{corollary}

\begin{proof}
It is sufficient to show
\[
|{\mathfrak D}_{k+1}(p)| = \frac{(k+1)!}{\left(\prod_{i=1}^t p_i\right)\left(\prod_{i=2}^{k+1}V_p(i)!\right)}
\] 
for $p = (p_1, \dots, p_t) \in P_{\geq 2}(k+1)$.
 Let $\Delta : \mathfrak{S}_{k+1} \to {\mathfrak D}_{k+1}(p)$ by
\[
\Delta(\sigma) = (\sigma(1), \sigma(2), \dots, \sigma(p_1))(\sigma(p_1+1),\dots,\sigma(p_1 +p_2))\dots.
\]
Note that $\Delta$ is surjective.  Since a cycle of $\Delta(\sigma)$ can be written with any one of its $p_i$ elements first and there are $\prod_{i=2}^{k+1}V_p(i)!$ linear orderings of the cycles by non-increasing length we have for $\delta \in \mathfrak{D}_{k+1}$
\[
|\Delta^{-1}(\delta)| =\left(\prod_{i=1}^t p_i\right)\left(\prod_{i=2}^{k+1}V_p(i)!\right).
\]
In particular, $|\Delta^{-1}(\delta)|$ is constant for $\delta \in \mathfrak{D}_{k+1}(p)$ so that
\[
|\mathfrak{S}_{k+1}| = \sum_{\delta \in {\mathfrak D}_{k+1}(p)} |\Delta^{-1}(\delta)|
\]
which implies
\[
|\mathfrak{D}_{k+1}(p)| = \frac{|\mathfrak{S}_{k+1}|}{|\Delta^{-1}(\delta)|}.
\]
\end{proof}

\begin{remark}
Corollary \ref{C:improvedCk} reduces the number of summands in the computation of $C_k$ exponentially because
\[
\log |P_{{\geq}(2)}(n)| \leq \log|P(n)|\approx \pi\sqrt{2n/3},
\]
but $\log|{\mathfrak D}_{n+1}| \approx n \log n$.
\end{remark}

We have computed the first few values of $C_k$ to be
\begin{align*}
C_6 &= 2092206 \\
C_7 &= 220611384 \\
C_8 &= 31373370936 \\
C_9 &= 5785037767440 \\
C_{10} &= 1342136211324090 \\
\vdots &\\
C_{100} &= 3433452419824795908447767175
863463034 52689609890358711139013913758 \\ 
& 7799578881707167888656395980536429532089 29209278848309297069686374206 \\ 
& 6180314961018984853143002532488553340756095279 15686375386625810970778 \\
& 8141825460673693192753149464456033881155778923 54872286012782651661555 \\
& 3106527369037122060186686535415242639036685247 99914172228056595466145 \\
& 2080249009900
\end{align*}
so that $C_{100} \approx 3.433... \cdot 10^{343}$, see A320653 \cite{OEIS}.  Note that $C_2 = 2$ gives the well-known result that, for a graph $G$ 
\[
c_3(G)=-2(\# \text{ of triangles in $G$}).
\]
We conclude by presenting the asymptotics of $C_k$.
\begin{theorem}
\label{T:simplexAsymptotics}
$C_k \sim (k+1)!k^{k+1}$ so $C_k = \exp(k\log k (2 + o(1))$.
\end{theorem}

\begin{proof}
As $\sigma$ is a derangement it follows that the length of each cycle is at least 2.  Notice
\[
|{\mathfrak D}_{k+1}|(k^2-1)^{\frac{k+1}{2}} \leq \sum_{\sigma = c_1c_2\dots c_t \in {\mathfrak D}_{k+1}} \left(\prod_{i=1}^t k^{l_i} + (-1)^{l_i + 1}\right) \leq |{\mathfrak D}_{k+1}|(k^3-1)^{\frac{k+1}{e}} 
\]
which implies
\[
\sum_{\sigma = c_1c_2\dots c_t \in {\mathfrak D}_{k+1}} \left(\prod_{i=1}^t k^{l_i} + (-1)^{l_i + 1}\right) \sim \frac{(k+1)!k^{k+1}}{e}.
\]
Thus,
\[
\lim_{k \to \infty} \frac{C_k}{(k+1)!k^{k+1}} = \lim_{k \to \infty} \frac{|{\mathfrak D}_{k+1}|k^{k+1}}{(k+1)! k^{k+1}} = \frac{1}{e}
\]
and we have  $C_k \sim (k+1)!k^{k+1}$.
\end{proof}

\section{Low Codegree Coefficients of $3$-Graphs}

In this section we provide an explicit formula for the first six leading coefficients of the characteristic polynomial of a $3$-graph.  We further apply these formulas to an example and show how this information can be used to determine the complete spectrum of a hypergraph. 

Let ${\cal H}$ be a simple $3$-graph with $n$ vertices.  We write $(\#\, H \in {\cal H}) = |\{S \subseteq {\cal H} : S \cong \underline{H})|$. From \cite{Coo} we have that $c_1 = 0, c_2 = 0$,
\[
c_3 = -3\cdot2^{n-3}(\# \, e \in {\cal H}),
\]
where $e = K_3^{(3)}$ is the single-edge hypergraph and \[
c_4 = 21\cdot2^{n-3}(\# \, K_4^{(3)} \in {\cal H}).
\]

We use Theorem \ref{T:codegreeFormula} to situation these results and further provide an analogous description of $c_5$ and $c_6$.  Clearly there are no Veblen 3-graphs with one or two edges as each vertex must have degree at least three.  It follows that $c_1,c_2 = 0$.  There is only one Veblen $3$-graph with three edges, the single edge with multiplicity three.  By Theorem \ref{T:codegreeFormula}, then, we have
\[
c_3 = -2^n\cdot \frac{3}{8}(\# \, e \in {\cal H}).
\]
Similarly, \cite{Coo} observed that the only Veblen 3-graph with four edges is the complete graph, which implies 
\[
c_4 = -2^n\cdot\frac{21}{8}(\# \, K_4^{(3)} \in {\cal H}).
\]
For the case of $c_5$ we have
\[
c_5 = -2^n\left(\frac{51}{16}(\# \,\Gamma_{5,1} \in {\cal H}) + \frac{27}{16}(\# \,\Gamma_{5,2} \in{\cal H})\right)
\]
where $\Gamma_{5,1}$ is the tight 5-cycle and $\Gamma_{5,2}$ is the 3-pointed crown (given explicitly in Figure \ref{T:1}).
We further compute
\begin{align*}
c_6 &= 2^{2n} \cdot \frac{9}{64}\left(\frac{(\# \, e \in {\cal H})^2}{2}\right) -2^n\left( \frac{3}{16}  
(\# \, e \in {\cal H}) + \frac{9}{8}(\# \, 
\Gamma_{6,1} \in {\cal H}) + \frac{9}{32} (\# \,\Gamma_{6,2} \in {\cal H}) \right .\\
&+ \frac{99}{32}\cdot 2 (\# \,\Gamma_{6,3} \in {\cal H}) + \frac{213}{16}(\# \,\Gamma_{6,4} \in {\cal H}) + \frac{69}{16}(\# \,\Gamma_{6,5} \in {\cal H}) + \frac{63}{32}(\# \,\Gamma_{6,6} \in {\cal H})\\
&\left . + \frac{129}{32}(\# \,\Gamma_{6,7} \in {\cal H}) + \frac{27}{32}\cdot 2(\# \,\Gamma_{6,8} \in {\cal H}) + \frac{63}{16}(\# \,\Gamma_{6,9} \in {\cal H}) + \frac{117}{32} (\# \,\Gamma_{6,10} \in {\cal H}) \right )
\end{align*} 
where $\Gamma_{6,i}$ are provided in Figure \ref{T:1}.

\begin{figure}[ht]
\begin{center}
 \begin{tabular}{|c|c|c|c|} 
 \hline
 $\Gamma$ & $E(\Gamma)$& $C_{\Gamma}$ & $|\Aut(\underline{\Gamma})|/|\Aut(\Gamma)|$ \\ 
 \hline\hline$\Gamma_{5,1}$ & $(123)(125)(145)(234)(345)$ & 51/16 & 1 \\ 
 \hline
  $\Gamma_{5,2}$ & $(123)(145)(145)(234)(235)$ & 27/16 & 1 \\ 
  \hline
  $\Gamma_{6,1}$ & $(123)^3(124)^3$ & 9/8 & 1 \\
  \hline
  $\Gamma_{6,2}$ & $(123)^3 (145)^3$ & 9/32 & 1 \\ 
 \hline 
 $\Gamma_{6,3}$ & $(123)^2(124)(135)(145)^2$ & 99/32 & 2 \\ 
  \hline 
 $\Gamma_{6,4}$ & $(123)(124)(125)(134)(135)(145)$ &213/16 & 1\\ 
   \hline 
 $\Gamma_{6,5}$ & $(123)(124)(156)(256)(345)(346)$ &69/16 & 1\\ 
  \hline 
 $\Gamma_{6,6}$ & $(123)(124)(145)(246)^3$ &63/32 & 1\\ 
 \hline
 $\Gamma_{6,7}$ & $(123)(134)(145)(246)(256)^2$ & 129/32 & 1\\ 
  \hline
 $\Gamma_{6,8}$ & $(123)^2(124)(356)(456)^2$ &27/32 & 2\\ 
 \hline
  $\Gamma_{6,9}$ & $(123)(124)(134)(256)(356)(456)$ &63/16 & 1\\ 
 \hline
  $\Gamma_{6,10}$ & $(123)(124)(135)(246)(356)(456)$ &117/16 & 1\\ 
 \hline
  $\Gamma_{9,2}$ & $(123)^6(145)^3$ & $9/32$ & 2 \\ 
  \hline
  $\Gamma_{9,3}$ & $(123)^3(145)^3(246)^3$ & 9/8 & 1 \\
  \hline
  $\Gamma_{9,4}$ & $(123)^3(145)^3(167)^3$ & $81/128$ & 1\\
  \hline
  $\Gamma_{12,1}$ & $(123)^9(145)^3$ & $9/32$ &2\\
  \hline
  $\Gamma_{12,2}$ & $(123)^6(145)^6$ & $27/64$ &1\\
  \hline
   $\Gamma_{12,3}$ & $(123)^6(145)^3(167)^3$ & $81/128$ &3\\
  \hline
   $\Gamma_{12,4}$ & $(123)^6(145)^3(246)^3$ & $63/32$ &3\\
  \hline
     $\Gamma_{12,5}$ & $(123)^3(145)^3(167)^3(246)^3$ & $459/64$ &1\\
  \hline
       $\Gamma_{12,6}$ & $(123)^3(145)^3(246)^3(356)^3$ & $255/16$ &1\\
  \hline
\end{tabular}
\caption{\label{T:1}Some connected Veblen 3-graphs and corresponding values}
\end{center}
\end{figure}


Recall that ${\cal V}_d$ and ${\cal V}_d^*$ denote the number of connected and (possibly) disconnected Veblen 3-graphs with $d$ edges.  We have computed
\[
(|{\cal V}_d|)_{d =1}^\infty = (0,0,1,1,2,11,26,122,781, \dots)
\]
and further 
\[
(|{\cal V}_d^*|)_{d =1}^\infty = (0,0,1,1,2,12,27,125,795, \dots)
\]
see A320648 \cite{OEIS}. 

\begin{remark}
The Fano Plane is a Veblen 3-graph with seven edges. The associated coefficient of the Fano Plane is $87/16 = 5.4375$.
\end{remark}
Note that, for the classical Harary-Sachs Theorem, the number of elementary graphs one needs to sum over for the codegree-$d$ coefficient is simply the number of partitions of $d$ into positive parts.  The number of Veblen 3-graphs is exponentially larger.

To demonstrate Theorem \ref{T:codegreeFormula} we have computed the first sixteen coefficients of the characteristic polynomial of the Fano Plane with two edges removed (called the Rowling hypergraph in \cite{Cla}), the Fano Plane with one edge removed, and the Fano Plane in Figure \ref{T:Fano}.  Aside from the novelty of performing these computations, the coefficients can be used to determine the characteristic polynomial of a given hypergraph \cite{Cla}.  
Consider the Rowling hypergraph
\[
{\mathcal R} = ([7], \{1,2,3\},\{\{1,4,5\},\{1,6,7\},\{2,5,6\},\{3,5,7\}\}).
\]
Using the aforementioned algorithm we were able to compute 
\begin{align*}
    \phi({\mathcal R}) = & x^{133}(x^3-1)^{27}(x^{15}-13x^{12}+65x^9-147x^6+157x^3-64)^{12} \\
    &\cdot (x^6-x^3+2)^6(x^6-17x^3+64)^3
\end{align*}
when traditional methods (i.e., direct computation of the resultant) failed to do so.  In this case, we can determine $\phi({\cal R})$ if we know $c_3,c_6, c_9, c_{12}$.  We show
\begin{align*}
    c_3 &= -240 \\
    c_6 &= 28320 \\
    c_9 &= -2190860 \\
    c_{12} &= 125012034.
\end{align*}
One can show that every Veblen infragraph of ${\cal R}$ has edge multiplicity congruent to 0 modulo 3 (this is immediate as the roots of $\phi({\cal R})$ are invariant under multiplication by third roots of unity \cite{Coo}).  Appealing to our previous formulae we have
\[
c_3 = -2^7\cdot\frac{3}{8}(\# \, e \in {\cal H}) = -2^n\cdot\frac{3}{8}\cdot 5 = -240
\]
and
\begin{align*}
c_6 &= \frac{c_3^2}{2!} -2^n\left( \frac{3}{16}(\# \, e \in {\cal H})+ \frac{9}{32} (\# \,\Gamma_{6,2} \in {\cal H})\right)
\\&=\frac{(-240)^2}{2} -2^7\left( \frac{3}{16}\cdot 5 + \frac{9}{32} \binom{5}{2}\right) =28320.
\end{align*}
Similarly we have
\begin{align*}
c_9 &= \frac{c_3^3}{3!} + c_3c_6- 2^7 \left (\frac{1}{8} (\# \, e \in {\cal H})+\frac{9}{32}\cdot 2(\#\, \Gamma_{9,2} \in {\cal H}) + \frac{9}{8}(\#\, \Gamma_{9,3} \in {\cal H}) + \frac{81}{128}(\# \, \Gamma_{9,4} \in {\cal H})
\right )\\
&= -2190860
\end{align*}
and 
\begin{align*}
c_{12} &= \frac{c_3^4}{4!} + \frac{c_3^2c_6}{2!} + c_3c_9 + 2^7\bigg(\frac{3}{32}(\# \, e \in {\cal H}) + \frac{9}{32}\cdot 2(\# \Gamma_{12,1} \in {\cal H}) \frac{27}{64} (\# \Gamma_{12,2} \in {\cal H}) \\
&+ \frac{81}{128}\cdot 3 (\# \Gamma_{12,3} \in {\cal H}) + \frac{63}{32}\cdot 3(\# \Gamma_{12,4} \in {\cal H}) + \frac{459}{64}(\# \Gamma_{12,5} \in {\cal H}) + \frac{255}{16}(\# \Gamma_{12,6} \in {\cal H})\bigg)\\
&=125012034.
\end{align*}

\begin{figure}[ht]
\begin{center}
\begin{tabular}{ |c|c|c|c| }
\hline
&FP-2 & FP-1 & FP \\
\hline
$c_0$ & 1 & 1 & 1 \\ 
\hline
$c_1$ & 0 & 0 & 0 \\ 
\hline
$c_2$ & 0 & 0 & 0  \\ 
\hline
$c_3$ & -240 & -288 & -336 \\
\hline
$c_4$ & 0 & 0 & 0  \\ 
\hline
$c_5
$ & 0 & 0 & 0  \\ 
\hline
$c_6$ & 28320 & 40788 &  55524\\
\hline
$c_7$ & 0 & 0 &  -696\\
\hline
$c_8$ & 0 & 0 & 0  \\ 
\hline
$c_{9}$ & -2190860 & -3788016 & -6017746\\
\hline
$c_{10}$ &  0  & 0& 220038\\
\hline
$c_{11}$ & 0 & 0 & 0  \\ 
\hline
$c_{12}$ &  125012034  & 259553826& 481293561\\
\hline
$c_{13}$ & 0 & 0 & -34237560  \\ 
\hline
$c_{14}$ & 0 & 0 & -122004  \\ 
\hline
$c_{15}$ & 5612445168 & -13997317932 & -30303162330\\
\hline
\end{tabular}
\end{center}
\caption{\label{T:Fano}
Leading coefficients of the Fano plane and its subgraphs.}
\end{figure}

\begin{figure}[ht]
\begin{center}
\includegraphics[scale = .5]{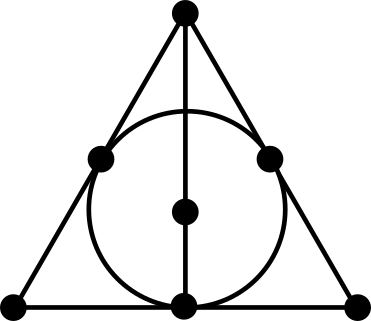} 
\hspace{1cm}
\includegraphics[scale = .5]{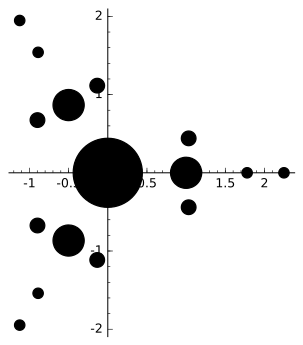} 
\end{center}
\caption{\label{F:2}The Rowling hypergraph and its spectrum, visualized in the complex plane by centering a disk at each eigenvalues whose area is proportional to its multiplicity.}
\end{figure}

\section{Open Questions}

A characterization of the multiplicity of the zero-eigenvalue for the adjacency characteristic polynomial of a graph remains open.  For convenience let $m_0$ denote the multiplicity of the zero-eigenvalue for a given polynomial.   Notice that one can provide an upper bound on $m_0$ by showing that a particular coefficient of $\phi({\cal H})$ is non-zero.  Combining this idea with the Harary-Sachs theorem gives the only known result in this direction for the adjacency characteristic polynomial: if $T$ is a (2-uniform) tree then $m_0$ is the size of the largest matching of $T$.  The authors ask if a similar result holds true for hypergraphs.  

In our proof of Lemma \ref{L:HS} we showed that, for 2-graphs, the summands in $c_d(G)$ arising from Veblen graphs which are not elementary graphs necessarily summed to zero.  We define the \emph{coefficient threshold} of $\phi({\cal H})$ as the least co-degree at which the coefficients of $\phi({\cal H})$ cancel thusly:

\begin{definition}
For an integer $v \geq 0$, the \emph{coefficient $v$-threshold} of a $k$-graph ${\cal H}$, denoted $\Th_v({\cal H})$, is the least integer such that for $d > \Th_v({\cal H})$
\[
\sum_{H \in {{\cal V}_d^*}({\cal H})}(-((k-1)^v)^{c(H)})C_H(\# H \subseteq {\cal H}) = 0.
\]
\end{definition}

Note that the contribution of $\cal{H}$ to the codegree-$d$ coefficient of $\phi(\cal{G})$ for ${\cal H} \subseteq \cal{G}$ is zero if $d > \Th_v(\cal{H})$ where $v = |V(\cal{G})|$.  As an example, we show that the $v$-threshold of the 3-uniform edge is $9\cdot 2^{v-3}$.

\begin{lemma}
\label{L:thresh}
Let $e$ be the 3-uniform edge.  Then $\Th_v(e) = 9\cdot2^{v-3}$ for $v \geq 3$.
\end{lemma}

\begin{proof}
For $n \geq 0$ define $f_n(t) : \mathbb{Z}^+ \to \mathbb{Q}$ by $f_n(0)=1$ and 
\[
f_n(t) = \sum_{H \in {{\cal V}^*_{3t}}(e)}(-(2^n)^{c(H)})C_H(\# H \subseteq e), t > 0.
\]
Observe that $f_3(t) = \phi_{3t}(e)$ by Theorem \ref{T:codegreeFormula}.  We conclude by showing $f_n(t) = (-1)^t\binom{3\cdot2^{n-3}}{t}.$ Considering the characteristic polynomial of a single $3$-uniform edge,
\[
f_3(0) = 1,\; f_3(1) = -3, \;f_3(2) = 3,\; f_3(3) = -1, \;f_3(4) = 0
\]
and $f_3(t) = 0$ for $t > 4$.  Indeed $f_3(t) = (-1)^t \binom{3}{t}$ for all $t$.    Suppose that for all $n$, up to some fixed $n$, we have  $f_n(t) = (-1)^t\binom{3\cdot2^{n-3}}{t}$ for all $t$.  Consider $f_{n+1}(t)$.  We claim 
\[
f_{n+1}(t) = \sum_{j=0}^tf_n(j)f_n(t-j).
\]
 Recall that the associated coefficient is multiplicative over components.  It follows that
\begin{align*}
\sum_{j=0}^tf_n(j)f_n(t-j) &= \sum_{H \in {\cal V}_3^*(e)} \left( \sum_{\substack{H =H_1 \cup H_2 \\H_1 \cap H_2 = \emptyset}} (-(2^n))^{c(H_1)}C_{H_1}\cdot(-(2^n))^{c(H_2)}C_{H_2}\right)\\
&=\sum_{H \in {\cal V}_3^*(e)} 2^n (-(2^n))^{c(H)}C_H = f_{n+1}(t).
\end{align*}
We have
\[
f_{n+1}(t) = (-1)^t \sum_{j=0}^t \binom{3\cdot 2^{n-3}}{j}\binom{3\cdot2^{n-3}}{t-j} = (-1)^t\binom{3\cdot2^{n-2}}{t}
\]
where the first equality follows from the inductive hypothesis and the second equality is given by the Chu-Vandermonde identity \cite{Chu}.  As $\Th_v(e) = f_v(t)$ we have that $\Th_v(e) = 9 \cdot2^{n-3}$ as $f_v(3\cdot2^{n-3}) = \pm 1$ for $t = 3\cdot2^{n-2}$ and $f_v = 0$ for $t > 3\cdot2^{n-2}$. 
\end{proof}

\begin{conjecture}
If ${\cal H} \subseteq {\cal G}$ are $k$-graphs, with $k > 2$ where $|V({\cal G})| = n$ then $\Th_n({\cal H}) \leq \Th_n({\cal G})$.
\end{conjecture}
This conjecture implies that the multiplicity of the zero-eigenvalue is at most $\deg(\phi({\cal G})) - \Th_n({\cal H})$.  Note that the restriction of $k > 2$ is necessary as the conjecture is not true for graphs.  For example, $C_6 \subseteq K_{6,6}$ and $\Th(K_{6,6}) = 2< \Th(C_6) = 6$.  More generally, it would help our understanding of the multiplicity of 0 to have a better understanding of $\Th_v(\mathcal{H})$, and so we ask:
\begin{question} Show how to compute or estimate $\Th_v(\mathcal{H})$ for various hypergraphs $\mathcal{H}$.
\end{question}

\bibliography{main.bib}{}
\bibliographystyle{plain}

\end{document}